\documentclass{amsart}
\usepackage{amsmath}
\usepackage{amsfonts}
\usepackage{graphics}
\usepackage{epsfig}
\usepackage{amssymb}
\usepackage{amscd}
\usepackage[all]{xy}
\usepackage{latexsym}
\usepackage{graphicx}
\usepackage{multirow}
\usepackage{geometry}

\theoremstyle{remark}
\newtheorem{lem}{\bf Lemma}[section]

\newtheorem{thm}{\bf Theorem}[section]
\newtheorem{cor}{\bf Corollary}[section]
\newtheorem{ex}{\bf Example}[section]

\newtheorem{prop}{\bf Proposition}[section]

\numberwithin{equation}{section} \numberwithin{figure}{section}

\renewcommand*{\to}{\rightarrow}
\renewcommand*{\bar}[1]{\overline{#1}}

\newcommand{\Aut}{\operatorname{Aut}}

\newcommand{\supp}{\operatorname{supp}}

\newcommand{\mb}[1]{\mathbb{#1}} 

\newcommand{\Hom}{\operatorname{Hom}}

\newcommand{\mc}[1]{\mathcal{#1}}
\newcommand{\mk}[1]{\mathfrak{#1}}

\newcommand{\Jac}{\operatorname{Jac}}

\newcommand{\dsp}{\displaystyle}
\newcommand{\mf}[1]{\mathbf{#1}}
\newcommand{\vol}{\operatorname{vol}}
\newcommand{\Div}{\operatorname{Div}}
\newcommand{\Eff}{\operatorname{Eff}}

\title[Mean field Equations on hyperelliptic curves]{Explicit Solutions to the mean field equations on hyperelliptic Curves of genus two}

\author{Jia-Ming (Frank) Liou}
\address{Department of Mathematics\\
National Cheng Kung University, Taiwan\\
fjmliou@mail.ncku.edu.tw}
\address{NCTS, Mathematics}

\begin{document}
\large 
\setcounter{section}{0}
\maketitle

\begin{abstract}\large 
Let $X$ be a complex hyperelliptic curve of genus two equipped with the canonical metric $ds^{2}$. We study mean field equations on complex hyperelliptic curves and show that the Gaussian curvature function of $(X,ds^{2})$ determines an explicit solution to a mean field equation.
\end{abstract}


\section{Introduction}
Mean field equations came originally from the study of prescribing (Gaussian) curvature problems in differential geometry. In \cite{LW1}, \cite{LW2}, Lin and Wang studied the mean field equation of the following type
\begin{equation}\label{MFELW}
\Delta u+\rho e^{u}=\rho\delta_{0},\quad\rho\in\mb R_{+}
\end{equation}
on a flat torus $T$ where $\Delta$ is the Laplace-Beltrami operator on $T.$ They discovered in \cite{LW1} that when $\rho=8\pi,$ (\ref{MFELW}) has a solution if and only if the set of critical points of the Green's function on $T$ contains points other than the three half-period points. In a recent paper \cite{LW3}, Chai and Lin and Wang showed that when $\rho=4\pi(2n+1),$ where $n$ is nonnegative integer, the number of solutions to (\ref{MFELW}) is $n+1$ except for a finite number of conformal isomorphism classes of flat tori and when $\rho=8\pi n,$ where $n$ is a positive integer, the solvability of (\ref{MFELW}) depends on the moduli space of flat tori. In this article, we consider the following mean field equation
\begin{equation}\label{MFEF}
\Delta u+\rho e^{u}=4\pi\sum_{i=1}^{m}n_{i}\delta_{P_{i}}
\end{equation}
on a compact Riemann surface $X$ of genus two with a hermitian metric $ds^{2},$ where $\rho$ and $n_{i}$ are some integers and $\{P_{i}\}$ are distinct points on $X.$ When $ds^{2}$ is the canonical metric on $X,$ we can prove that the Gaussian curvature function of $(X,ds^{2})$ determines an explicit solution to (\ref{MFEF}). In this case, $m=6$ and $n_{i}=1$ and $P_{i}$ are the Weierstrass points of $X$ for $1\leq i\leq 6.$ Furthermore, we also discover that given any finite subset of distinct points $\{Q_{1},\cdots,Q_{s}\}$ of $X$ and any finite set of natural numbers $\{m_{1},\cdots,m_{s}\},$ for a certain choice of nonnegative continuous function $h$ on $X,$ there exists a real valued function $v$ defined and smooth on $X\setminus \{Q_{1},\cdots,Q_{s},P_{1},\cdots,P_{6}\}$ so that $v$ satisfies the following mean field equation
\begin{equation}\label{MFEF2}
\Delta v(x)+\rho h(x)e^{v(x)}=4\pi\sum_{i=1}^{6}\delta_{P_{i}}-2\pi\sum_{j=1}^{s}m_{j}(\delta_{Q_{j}}+\delta_{\iota(Q_{j})})
\end{equation}
where $\iota:X\to X$ is an involution on $X$ and $\rho$ is some integer. The existence of the involution on $X$ is due to the fact that any genus two compact Riemann surface is hyperelliptic. Let us formulate the precise statements of our main results as follows.

Let $X$ be a compact Riemann surface\footnote{All the Riemann surfaces in this paper are assumed to be connected.} of genus $g\geq 2$ and $ds^{2}$ be any hermitian metric on $X.$ 
The Gauss-Bonnet theorem tells us that
$$\int_{X}Kd\nu=2\pi\chi(X)=4\pi(1-g)<0,$$
where $K$ is the Gaussian curvature function of $(X,ds^{2})$ and $d\nu$ is the volume form of $(X,ds^{2}).$ Then the open set $U=\{P\in X:K(P)<0\}$ is nonempty. Let $\mu:X\to \Jac(X)$ be the Abel-Jacobi map and $d\widetilde{s}^{2}$ be the standard flat metric on $\Jac(X).$ The canonical metric on $X$ is the pull back metric $\mu^{*}d\widetilde{s}^{2}.$ Now let $ds^{2}$ be the canonical metric on $X$ and $K$ be the corresponding Gaussian curvature function. It is not difficult to show that $K$ is nonpositive on $X,$ for example, see Section \ref{ab}, or \cite{JL}. Let $Z$ be the set of zeros of $K.$ Then $U=X\setminus Z.$ In \cite{JL}, Lewittes shows that $Z$ is nonempty if and only if $X$ is a hyperelliptic Riemann surface and $Z$ is the set of all Weierstrass points of $X.$ As a consequence, the subset $Z$ of $X$ is either empty or a finite subset of $X.$ If $X$ is not a hyperelliptic Riemann surface, $U=X.$ When $X$ is a compact Riemann surface of genus two, $X$ is hyperelliptic. Then we have the following result:


\begin{thm}\label{thm2}
Let $(X,ds^{2})$ be a compact Riemann surface of genus $g=2$ where $ds^{2}$ is the canonical metric on $X.$ Define a function $u:X\setminus Z\to\mb R$ by $u=\log(-K),$ where $Z=\{P_{1},\cdots,P_{6}\}$ is the set of all Weierstrass points of $X.$ Then $u$ is smooth on $X\setminus Z$ and satisfies the following mean field equation:
\begin{equation}\label{e3}
\Delta u+6e^{u}=4\pi\sum_{i=1}^{6}\delta_{P_{i}}.
\end{equation}
Here $\delta_{P}:C^{\infty}(X)\to\mb R$ is the Dirac delta function defined by $\delta_{P}(\varphi)=\varphi(P)$ for any $\varphi\in C^{\infty}(X)$ and for any $P\in X.$
\end{thm}




Since any hyperelliptic Riemann surface of genus $g$ is conformally equivalent to the compactification of an affine plane curve $y^{2}=f(x)$ over $\mb C,$ where $f(x)$ is a complex polynomial of degree $2g+1$ or $2g+2.$ The compactification of the complex affine plane curve $y^{2}=f(x)$ is a complex smooth projective curve, called a complex hyperelliptic curve. We will prove Theorem \ref{thm2} in the case when $X$ is a complex hyperellptic curve defined by the polynomial equation $y^{2}=f(x).$ 

Let $X$ be the complex hyperelliptic curve of genus $g$ defined by $y^{2}=f(x).$ For every point $P\in X,$ we will construct a nonnegative continuous function $F_{P}:X\to\mb R$ in Section \ref{mf} such that the zero set of $F_{P}$ is $\{P\}$ and $F_{P}$ is smooth on $X\setminus\{P\}.$ If $D=\sum_{P\in X}n_{P}P$ is a divisor on $X,$ we define $F_{D}=\prod_{P\in X}F_{P}^{n_{P}}.$ Then $F_{D}$ is smooth and nonnegative on $X\setminus\supp(D).$ Let $\Eff(X)$ be the vector subspace of $C(X;\mb C),$ the space of complex valued continuous functions on $X,$ spanned by $1$ and $F_{D}$ where $D$ runs through all effective divisors on $X.$ We discover that $\Eff(X)$ is a unital filtered $*$-subalgebra of $C(X,\mb C)$ whose closure in $C(X,\mb C)$ is isomorphic to the algebra of complex valued continuous functions on the one dimensional complex projective space $\mb P^{1}$ (with respect to the complex analytic topology).

\begin{thm}\label{thmeff}
Let $X$ be a complex hyperelliptic curve of genus $2$ and $\{P_{1},\cdots,P_{6}\}$ be the set of all Weierstrass points of $X.$ Given any set of positive integers $\{m_{1},\cdots,m_{s}\}$ and any finite set $\{Q_{1},\cdots,Q_{s}\}$ of distinct points on $X$ such that $P_{i}\not\in\{Q_{1},\cdots,Q_{s}\}$ for $1\leq i\leq 6,$ the following mean field equation
\begin{equation}\label{MFELW2}
\Delta v(x)+2\left(6-\sum_{i=1}^{s}m_{i}\right)F_{D}(x)e^{v(x)}=4\pi\sum_{i=1}^{6}\delta_{P_{i}}-2\pi\sum_{j=1}^{s}m_{j}(\delta_{Q_{j}}+\delta_{\iota(Q_{j})})
\end{equation}
on $X$ always posses a solution $v$ defined and smooth on $X\setminus\{P_{1},\cdots,P_{6},Q_{1},\cdots,Q_{s}\},$ where $D$ is the divisor on $X$ defined by $D=\sum_{i=1}^{s}m_{i}Q_{i}$ and $\iota:X\to X$ is the hyperelliptic involution on $X.$
\end{thm}

{\bf Acknowledgements} We are indebted to Prof. C.S. Lin and Prof. C. L. Wang from the National Taiwan university for introducing us the topic of mean field equations. Without them, this paper would not appear. The author was partially supported by MOST 104-2115-M-006-014 Grant, and by the Headquarters of University Advancement at National Cheng Kung University, which is sponsored the Ministry of Education, Taiwan, ROC, and by NCTS. The author also would like
to thank the referees for their valuable comments which helped to improve the manuscript. 

\section{The Gaussian curvature function of the canonical metric on Compact Riemann surfaces}\label{ab}

In this section, we will briefly review some materials in \cite{JL} which we need in this paper. Let $X$ be a compact Riemann surface of genus $g\geq 2$ and $\Omega_{X}$ be the sheaf of holomorphic one forms on $X.$ The space $H^{0}(X,\Omega_{X})$ of global sections of $\Omega_{X}$ is a complex vector space of dimension $g$ while the integral homology group $H_{1}(X,\mb Z)$ of $X$ is a free abelian group of rank $2g.$ Choose a symplectic $\mb Z$-basis $\{a_{1},\cdots,a_{g},b_{1},\cdots,b_{g}\}$ for the integral homology group $H_{1}(X,\mb Z)$ of $X$ with respect to the intersection form on $H_{1}(X,\mb Z)$ and choose a complex basis $\{\omega_{1},\cdots,\omega_{g}\}$ for $H^{0}(X,\Omega_{X}).$ For each $1\leq i,j\leq g,$ we denote by
$$\int_{a_{j}}\omega_{i}=\alpha_{ij},\quad \int_{b_{j}}\omega_{i}=\beta_{ij}.$$
Denote the $g\times g$ complex matrices $(\alpha_{ij})_{1\leq i,j\leq g}$ and $(\beta_{ij})_{1\leq i,j\leq g}$ by $A$ and $B$ respectively. The $g\times 2g$ complex matrix $(A,B)$ is called a period matrix for $X.$ When $A$ is the $g\times g$ identity matrix $I_{g}=(\delta_{ij})_{1\leq i,j\leq g},$ the period matrix is said to be normalized with respect to the basis $\{a_{1},\cdots,a_{g},b_{1},\cdots,b_{g}\}.$ When the period matrix is normalized, the basis $\{\omega_{1},\cdots,\omega_{g}\}$ forms an orthonormal basis for $H^{0}(X,\Omega_{X})$ with respect to the hermitian inner product 
\begin{equation}\label{inn}
\langle \omega,\eta\rangle_{H^{0}(X,\Omega_{X})}=\frac{i}{2}\int_{X}\omega\wedge\bar{\eta}.
\end{equation}
Let $\Lambda$ be the lattice of $\mb C^{g}$ generated by the column vectors of the period matrix $(A,B).$ The Jacobian variety $\Jac(X)$ is the complex torus $\Jac(X)=\mb C^{g}/\Lambda.$ Let $P_{0}$ be a based point of $X.$ The Abel Jacobi map $\mu:X\to \Jac(X)$ is defined to be
$$\mu(P)=\left(\int_{P_{0}}^{P}\omega_{1},\cdots,\int_{P_{0}}^{P}\omega_{g}\right)\mod\Lambda.$$

Let $d\widetilde{s}_{H}^{2}=\sum_{i=1}^{g}h_{ij}dz_{i}\otimes d\bar{z}_{i}$ be a hermitian metric on $\mb C^{g}$ where $(z_{1},\cdots,z_{g})$ is the complex coordinate on $\mb C^{g}$ and $H=(h_{ij})$ is a $g\times g$ complex positive definite hermitian matrix. The hermitian metric $d\widetilde{s}_{H}^{2}$ on $\mb C^{g}$ induces a K\"{a}hler metric on $\Jac(X).$ We use the same notation $d\widetilde{s}_{H}^{2}$ for the induced metric of $d\widetilde{s}_{H}^{2}$ on $\Jac(X).$ The canonical metric $ds_{H}^{2}$ on $X$ is the pull back metric $\mu^{*}d\widetilde{s}_{H}^{2}$ via $\mu,$ i.e. $ds_{H}^{2}=\mu^{*}d\widetilde{s}_{H}^{2}.$ Since $(z_{i}\circ \mu)(P)=\int_{P_{0}}^{P}\omega_{i},$ by simple computation $\mu^{*}dz^{i}=d(z^{i}\circ \mu)=\omega_{i}.$ We find that 
$$ds_{H}^{2}=\sum_{i,j=1}^{n}h_{ij}\omega_{i}\otimes\overline{\omega_{j}}.$$ 

Let $(z_{\alpha},U_{\alpha})$ be a complex local chart on $X$ and $f_{i}(z_{\alpha})dz_{\alpha}$ be the local representation of $\omega_{i}$ with respect to $(z_{\alpha},U_{\alpha}).$ Let $\mf f_{\alpha}:U_{\alpha}\to \mb C^{g}$ be the holomorphic map $$\mf f(z_{\alpha})=(f_{1}(z_{\alpha}),\cdots,f_{n}(z_{\alpha})).$$ Then the local representation of $ds_{H}^{2}$ with respect to $(z_{\alpha},U_{\alpha})$ is given by $$ds_{H}^{2}=\langle \mf f_{\alpha},\mf f_{\alpha}\rangle_{H}dz_{\alpha}\otimes d\bar{z}_{\alpha},$$
where $\langle z,w\rangle_{H}=\sum_{i,j=1}^{g}h_{ij}z_{i}\bar{w}_{j}$ for any $z=(z_{1},\cdots,z_{g})$ and $w=(w_{1},\cdots,w_{g})$ in $\mb C^{g}.$ 
The Riemann Roch Theorem tells us that $\{\omega_{1},\cdots,\omega_{g}\}$ do not have common zeros. The function $\rho_{\alpha}=\langle \mf f_{\alpha},\mf f_{\alpha}\rangle_{H}$ on $U_{\alpha}$ is positive. By a simple computation, the Gaussian curvature function on $U_{\alpha}$ has the expression
$$K(z_{\alpha})=-\frac{1}{2}\Delta \log\rho_{\alpha}=-2\frac{S_{z_{\alpha}}(\rho_{\alpha})}{\rho_{\alpha}^{3}},$$
where $S_{z_{\alpha}}(\rho_{\alpha})=\langle \mf f_{\alpha}',\mf f_{\alpha}'\rangle_{H}\langle \mf f_{\alpha},\mf f_{\alpha}\rangle_{H}-\left|\langle \mf f_{\alpha}',\mf f_{\alpha}\rangle_{H}\right|^{2}.$
By Cauchy-Schwarz inequality on $(\mb C^{g},\langle\cdot,\cdot\rangle_{H}),$ $S_{z_{\alpha}}(\rho_{\alpha})$ is a nonnegative function on $U_{\alpha},$ i.e. $S_{z_{\alpha}}(\rho_{\alpha})\geq 0$ on $U_{\alpha}.$ As a result, the Gaussian curvature $K$ is a nonpositive function on $X.$ As we have mentioned in the introduction, Lewittes proved in \cite{JL} that $K$ is negative when $X$ is not a hyperelliptic Riemann surface. Now let us study the function $K$ when $X$ is a hyperelliptic Riemann surface. 

A compact Riemann surface of genus $g\geq 2$ is hyperelliptic if it admits a positive divisor $D$ of degree two such that $l(D)\geq 2$ or equivalently there exists a nonconstant meromorphic function $\pi$ in $L(D)$ such that $\pi:X\to\mb P^{1}$ is a degree two ramified covering map. A hyperelliptic Riemann surface is conformally equivalent to the compactification of the nonsingular affine plane curve
$$C_{0}=\{(x,y)\in\mb C^{2}:y^{2}=f(x)\}$$
over $\mb C$ where $f$ is a complex polynomial of degree $2g+1$ or $2g+2$ with distinct complex roots. In this section, we will consider the case when $\deg f=2g+2$ and review some basic facts from \cite{Mum} describing the compactification of $C_{0}.$ (The proof of Theorem \ref{thm2} in the case when $\deg f=2g+1$ is the same as that in the case when $\deg f=2g+2.$) Suppose that $f(x)=\prod_{k=1}^{2g+2}(x-e_{i})$ where $\{e_{1},\cdots,e_{2g+2}\}$ are distinct complex numbers. Let $g(z)=\prod_{k=1}^{2g+2}(1-e_{i}z).$ Then $g(z)=z^{2g+2}f(1/z).$ Let $C_{0}'$ be the affine nonsingular plane curve $\{(z,w)\in\mb C^{2}:w^{2}=g(z)\}.$ Let $U_{0}$ be the affine open subset of $C_{0}$ consisting of points $(x,y)$ such that $x\neq 0$ and $U_{0}'$ be the affine open subset of $C_{0}'$ consisting of points $(z,w)$ such that $z\neq 0.$ There is an isomorphism $\varphi:U_{0}\to U_{0}'$ defined by $\varphi(x,y)=(1/x,y/x^{g+1}).$ The compactification $X$ of $C_{0}$ is the gluing of $C_{0}$ and $C_{0}'$ along the isomorphism $\varphi.$ We set $\infty_{\pm}=(0,\pm 1)$ on $C_{0}'.$ The degree two ramified covering map $\pi:X\to\mb P^{1}$ is given by
\begin{equation*}
\pi(P)=
\begin{cases}
(x(P):1)&\mbox{ if $P\in C_{0},$}\\
(1:z(P))&\mbox{ if $P\in C_{0}'.$}
\end{cases}
\end{equation*}
Here $(z_{0}:z_{1})$ is the homogenous coordinate on $\mb P^{1}.$
The Weierstrass points of $X$ are the $2g+2$ branched points $\{P_{1},\cdots,P_{2g+2}\}$ of $\pi$ such that $(x(P_{k}),y(P_{k}))=(e_{k},0)$ for $1\leq k\leq 2g+2.$ Denote $Z=\{P_{1},\cdots,P_{2g+2}\}.$ We will show that the set $Z$ is exactly the zero set of $K.$

It is well known that $H^{0}(X,\Omega_{X})$ consists of one form of the form $F(x)dx/y$ where $F$ is a polynomial in $x$ of degree at most $g-1$ and that $\{x^{i-1}dx/y:1\leq i\leq g\}$ forms a basis for $H^{0}(X,\Omega_{X}),$ for example, see \cite{Mum} and \cite{Sp}. We choose a symplectic homology $\mb Z$-basis $\{a_{1},\cdots,a_{g},b_{1},\cdots,b_{g}\}$ for $H_{1}(X,\mb Z)$ and choose a basis $\{\omega_{1},\cdots,\omega_{g}\}$ for $H^{0}(X,\Omega_{X}).$ This defines the Abel Jacobi map $\mu:X\to \Jac(X)$ after fixing a based point of $X.$ Let us write $\omega_{i}=\sum_{j=1}^{g}a_{ji}x^{j-1}dx/y$ for $1\leq i,j\leq g,$ where $(a_{ij})$ is a $g\times g$ complex invertible matrix. On $\Jac(X),$ we consider the standard flat metric $d\widetilde{s}^{2}=d\widetilde{s}_{H}^{2}$ where $H=I_{g}$ is the $g\times g$ identity matrix. Then the canonical metric on $X$ can be rewritten as
$$ds^{2}=\sum_{i=1}^{g}\omega_{i}\otimes\bar{\omega}_{i}=\left(\frac{1}{|y^{2}|}\sum_{k,j=1}^{g}c_{kj}x^{k-1}\bar{x}^{j-1}\right)dx\otimes d\bar{x},$$
where $C=(c_{kj})$ is a $g\times g$ positive definite hermitian matrix defined by $c_{kj}=\sum_{i=1}^{g}a_{ki}\bar{a}_{ji}$ for $1\leq k,j\leq g.$ When $P\neq\infty_{\pm}$ and $P\not\in W,$ the function $x$ defines a local coordinate\footnote{If we want a coordinate $\zeta$ around $P$ so that $\zeta(P)=0,$ you consider a change of coordinate $\zeta=x-x(P)$ in a neighborhood of $P.$} in an open neighborhood $U_{P}$ of $P.$ In this case, $\log |y^{2}|=\log|f(x)|$ is a nonzero harmonic function on $U_{P}.$ Define a holomorphic map $\mf f:U_{P}\to\mb C^{g}$ by $\mf f(x)=(1,x,\cdots,x^{g-1}).$ If
$$\rho(x)=\frac{\langle \mf f(x),\mf f(x)\rangle_{C}}{|y^{2}|},$$
then $ds^{2}=\rho(x)dx\otimes d\bar{x}.$ Since $\log|f(x)|$ is harmonic on $U_{P},$ we have
\begin{align*}
K(x)&=-\frac{2}{\rho(x)}\frac{\partial^{2}}{\partial x\partial\bar{x}}\log\rho\\
&=-\frac{2}{\rho(x)}\cdot\frac{\partial^{2}}{\partial x\partial\bar{x}}(\log \langle \mf f(x),\mf f(x)\rangle_{C}-\log|f(x)|)\\
&=-\frac{2|f(x)|}{ \langle \mf f(x),\mf f(x)\rangle_{C}^{3}}\cdot S(\mf f)(x)
\end{align*}
where $S(\mf f)(x)=\langle \mf f'(x),\mf f'(x)\rangle_{C}\langle \mf f(x),\mf f(x)\rangle_{C}-|\langle \mf f(x),\mf f'(x)\rangle_{C}|^{2}$ on $U_{P}.$ Since $$\mf f'(x)=(0,1,2x,\cdots,(g-1)x^{g-2})\mbox{ on $U_{P}$},$$ and is not parallel to $\mf f(x)=(1,x,x^{2},\cdots,x^{g-1})$ at any point of $U_{P}$ as a vector in $\mb C^{g},$ the function $S(\mf f)(x)$ is always positive on $U_{P}$ and hence $K$ is negative on $U_{P}.$ 

When $P=\infty_{\pm},$ in a neighborhood $U_{P}$ of $P=\infty_{+}$ or $P=\infty_{-},$ the function $z$ defines a local coordinate on $U_{P}.$ Define $\mf g:U_{P}\to\mb C^{g}$ by $\mf g(z)=(z^{g-1},z^{g-2},\cdots,z,1).$ If
$$\rho(z)=\frac{\langle\mf g(z),\mf g(z)\rangle_{C}}{|w^{2}|},$$
then $ds^{2}$ can be rewritten as $\rho(z)dz\otimes d\bar{z}.$ Similarly, the Gaussian curvature function on $U_{P}$ can be rewritten as
$$K(z)=-\frac{2|w^{2}|}{\langle \mf g(z),\mf g(z)\rangle_{C}^{3}}S(\mf g)(z).$$
Similar reason as above implies that $S(\mf g)$ is positive on $U_{P}$ for $P=\infty_{+}$ or $P=\infty_{-};$ hence $K<0$ on $U_{P}$ for $P=\infty_{+}$ or $P=\infty_{-}.$

In a neighborhood $U_{P_{k}}$ of $P_{k},$ we choose $\zeta=\sqrt{x-e_{k}}$ and define functions
$$\mf f_{k}:U_{P_{k}}\to\mb C^{g}\mbox{ by }\mf f_{k}(\zeta)=(2,2(\zeta^{2}+e_{k}),\cdots,2(\zeta^{2}+e_{k})^{g-1})$$
and $f_{k}:U_{p_{k}}\to\mb C$ by $f_{k}(\zeta)=\prod_{1\leq j\leq 2g+2, j\neq k}(\zeta^{2}+e_{k}-e_{j}).$ By definition. $f_{k}$ is a nonzero holomorphic function on $U_{P_{k}}.$ Let $\rho_{k}:U_{P_{k}}\to\mb R$ be the function $\rho_{k}(\zeta)=\langle \mf f_{k}(\zeta),\mf f_{k}(\zeta)\rangle_{C}/|f_{k}(\zeta)|.$ Then $ds^{2}$ can be rewritten as $\rho_{k}(\zeta)d\zeta\otimes d\bar{\zeta}.$ Notice that on $U_{P_{k}},$ the function $\log|f_{k}(\zeta)|$ is harmonic. Hence the Gaussian curvature has the following local expression
$$K(\zeta)=-2\frac{|f_{k}(\zeta)|}{\langle \mf f_{k}(\zeta),\mf f_{k}(\zeta)\rangle_{C}}\cdot \frac{S(\mf f_{k})(\zeta)}{\langle \mf f_{k}(\zeta),\mf f_{k}(\zeta)\rangle_{C}^{2}}$$
Notice that $\mf f_{k}'(\zeta)=4\zeta \mf g_{k}(\zeta)$ where $g_{k}(\zeta)=(0,1,2(\zeta^{2}+e_{k}),\cdots,(g-1)(\zeta^{2}+e_{k})^{g-1})$ on $U_{P_{k}}$ and hence
$$S(\mf f_{k})(\zeta)=4|\zeta|^{2}\left(\langle \mf g_{k}(\zeta),\mf g_{k}(\zeta)\rangle_{C}\langle \mf f_{k}(\zeta),\mf f_{k}(\zeta)\rangle_{C}-|\langle \mf f_{k}(\zeta),\mf g_{k}(\zeta)\rangle_{C}|^{2}\right)$$ 
Therefore we see that
$$K(\zeta)=-2|\zeta|^{2}\frac{|f_{k}(\zeta)|}{\langle \mf f_{k}(\zeta),\mf f_{k}(\zeta)\rangle_{C}^{3}}\cdot \widetilde{S}(\zeta),$$
where $\widetilde{S}(\zeta)=\langle \mf g_{k}(\zeta),\mf g_{k}(\zeta)\rangle_{C}\langle \mf f_{k}(\zeta),\mf f_{k}(\zeta)\rangle_{C}-|\langle \mf f_{k}(\zeta),\mf g_{k}(\zeta)\rangle_{C}|^{2}.$ 
Since $\mf f_{k}(\zeta)$ and $\mf g_{k}(\zeta)$ are not parallel at any point of $U_{P_{k}},$ Cauchy-Schwarz inequality implies that $\widetilde{S}(\zeta)>0$ on $U_{P_{k}}.$ Moreover, since $|f_{k}(\zeta)|$ is also positive on $U_{P_{k}}$ and $K(\zeta)=0$ if and only if $\zeta=0$ in $U_{P},$ $K$ has only one zero at $P_{k}$ in $U_{P_{k}}$ ($\zeta=0$ corresponds to the point $P_{k}$ on $C.$) We conclude that when $X$ is a complex hyperelliptic curve, the zero set $Z$ of the Gaussian curvature $K$ coincides with the set of ramification points of $\pi:X\to\mb P^{1},$ i.e. coincides with the set $Z$ of all Weierstrass points. 

Now we are ready to prove Theorem \ref{thm2}. Define $u:X\setminus Z\to\mb R$ by $u=\log(-K).$ Then $u$ is smooth and on the deleted neighborhood $U_{P_{k}}\setminus\{P_{k}\},$ the function $u$ has the local expression
$$u=\log(-K)=2\log|\zeta|+\log 2+\log|f_{k}(\zeta)|+\log \widetilde{S}(\zeta)-3\log \langle \mf f_{k}(\zeta),\mf f_{k}(\zeta)\rangle_{B}.$$
The classical analysis tells us that the action of $\Delta$ on $\log|\zeta|$ in a neighborhood of $P_{k}$ contributes a Dirac measure $2\pi\delta_{P_{k}}$ centered at $P_{k}.$ (This technique will be reviewed in a later section in order to prove a more general result.) If we denote $\Delta u+6e^{u}$ by $\Phi,$ then
\begin{equation}\label{e2}
\Delta u+6e^{u}=\Phi+4\pi\sum_{k=1}^{2g+2}\delta_{P_{k}}
\end{equation}
as a distribution on $X.$ 
When $g=2,$ $\Phi=0$ and thus (\ref{e2}) implies (\ref{MFEF}). In more detail, when $g=2,$ the conformal factor of the canonical metric on $X$ is given locally by
\begin{equation*}
\rho(\zeta)=
\begin{cases}
\dsp\frac{c_{11}+c_{12}\bar{\zeta}+c_{21}\zeta+c_{22}\zeta\bar{\zeta}}{|f(\zeta)|}&\mbox{ on $U_{P}$ for $P\not\in \{\infty_{\pm}\}\cup Z$}\\
\dsp\frac{c_{11}\zeta\bar{\zeta}+c_{12}\bar{\zeta}+c_{21}\zeta+c_{22}}{|g(\zeta)|}&\mbox{ on $U_{P}$ for $P\in\{\infty_{\pm}\}$}\\
\dsp\frac{c_{11}+c_{12}\bar{(\zeta^{2}+e_{k})}+c_{21}(\zeta^{2}+e_{k})+c_{22}|\zeta^{2}+e_{k}|^{2}}{|f_{k}(\zeta)|}&\mbox{ on $U_{P}$ for $P=P_{k}.$}
\end{cases}
\end{equation*}
and hence the corresponding Gaussian curvature function $K$ is given locally by 
\begin{equation*}
K(\zeta)=
\begin{cases}
\dsp \frac{-2|f(\zeta)|\det C}{(c_{11}+c_{12}\bar{\zeta}+c_{21}\zeta+c_{22}\zeta\bar{\zeta})^{3}}&\mbox{ on $U_{P}$ for $P\not\in \{\infty_{\pm}\}\cup Z$}\\
\dsp\frac{-2|g(\zeta)|\det C}{(c_{11}\zeta\bar{\zeta}+c_{12}\zeta+c_{21}\bar{\zeta}+c_{22})^{3}}&\mbox{ on $U_{P}$ for $P\in\{\infty_{\pm}\}$}\\
\dsp\frac{-2|\zeta|^{2}|f_{k}(\zeta)|\det C}{4^{3}(c_{11}+c_{12}\overline{(\zeta^{2}+e_{k})}+c_{21}(\zeta^{2}+e_{k})+c_{22}|(\zeta^{2}+e_{k})|^{2})^{3}}&\mbox{ on $U_{P}$ for $P=P_{k}.$}
\end{cases}
\end{equation*}
These formulae give us local expression of the function $u=\log(-K).$ By direct computation, we prove (\ref{e3}) or equivalently $\Phi=0.$

Let us prove that $\Phi\neq 0$ when $g>2.$ Let $P$ be a point on $X$ such that $x(P)=0$ and choose a coordinate neighborhood $U_{P}$ of $P.$ On $U_{P},$ we write $ds^{2}=\rho(x)dx\otimes d\bar{x}$ as before. The function $\Phi$ has the following local expression on $U_{P}:$
$$
\Phi=\Delta\log\phi=\frac{4}{\rho\phi^{2}}(\phi_{x\bar{x}}\phi-\phi_{x}\phi_{\bar{x}})
$$
where $\phi=\rho_{x\bar{x}}\rho-\rho_{x}\rho_{\bar{x}}.$ Let $C_{k}$ be the $k\times k$ principal submatrix $[c_{ij}]_{i,j=1}^{k}$ of $C$ for $1\leq k\leq g.$ After some elementary calculation, $\Phi(P)=16 \det C_{3}/\det C_{2}.$ Since $C$ is positive definite, $\det C_{k}>0$ for all $1\leq k\leq g$ by basic linear algebra. Therefore $\Phi(P)>0.$ We proved our assertion. 


\section{Mean field equations on hyperelliptic Curves}\label{mf}

In this section, we let $X$ be the complex hyperelliptic curve of genus $g$ defined by $y^{2}=f(x)$ as above in Section \ref{ab} and $ds^{2}$ be the canonical metric on $X$ defined by the Abel-Jacobi map $\mu:X\to\Jac(X),$ where $\mu$ is defined by the basis $\{\omega_{i}=x^{i-1}dx/y:1\leq i\leq g\}$ for $H^{0}(X,\Omega_{X})$ and by a(ny) symplectic basis for $H_{1}(X,\mb Z).$ In this case, the canonical metric $ds^{2}$ on $X$ has the following expression:
\begin{equation}\label{m1}
ds^{2}=\frac{1}{|y^{2}|}\left(\sum_{i=1}^{g}|x|^{2(i-1)}\right)dx\otimes d\bar{x}.
\end{equation}
Let $\sigma_{g}:[0,\infty)\to [0,\infty)$ be the function defined by $\sigma_{g}(t)=1+t+\cdots+t^{g-1}$ for $t\geq 0.$ Then the canonical metric can be rewritten as 
$$ds^{2}=\frac{\sigma_{g}(x\bar{x})}{|y^{2}|}dx\otimes d\bar{x}.$$
Using $\sigma_{g},$ we can construct a nonnegative continuous function $F_{P}:X\to\mb R$ for each $P\in X$ as follows.
For $P\in X\setminus\{\infty_{\pm}\},$ we define
$$F_{P}(Q)=
\begin{cases}
\dsp \frac{|x(Q)-x(P)|}{(\sigma_{g}(x(Q)\bar{x(Q)})(1+x(Q)\bar{x(Q)})^{2})^{\frac{1}{2g+2}}} &\mbox{ if $Q\in C_{0},$}\\
\dsp \frac{|1-z(Q)x(P)|}{(\sigma_{g}(z(Q)\bar{z(Q)})(1+z(Q)\bar{z(Q)})^{2})^{\frac{1}{2g+2}}} &\mbox{ if $Q\in C_{0}'$}
\end{cases}$$ 
and define $F_{\infty_{\pm}}:X\to\mb R$ by
$$F_{\infty_{\pm}}(Q)=
\begin{cases}
\dsp \frac{1}{(\sigma_{g}(x(Q)\bar{x(Q)})(1+x(Q)\bar{x(Q)})^{2})^{\frac{1}{2g+2}}} &\mbox{ if $x$ is a coordinate around $U_{Q}$ for $Q\in C_{0}$}\\
\dsp \frac{|z(Q)|}{(\sigma_{g}(z(Q)\bar{z(Q)})(1+z(Q)\bar{z(Q)})^{2})^{\frac{1}{2g+2}}} &\mbox{ if $z$ is a coordinate around $U_{Q}$ for $Q=\infty_{\pm}$}.
\end{cases}$$

\begin{ex}
When $g=2,$ $\sigma_{2}(t)=1+t.$ The function $F_{P}:X\to\mb R$ has the local expression
$$F_{P}(Q)=
\begin{cases}
\dsp \frac{|x(Q)-x(P)|}{\sqrt{1+x(Q)\bar{x(Q)}}} &\mbox{ if $Q\in C_{0},$}\\
\dsp \frac{|1-z(Q)x(P)|}{\sqrt{1+z(Q)\bar{z(Q)}}} &\mbox{ if $Q\in C_{0}'$}
\end{cases}$$ 
and define $F_{\infty_{\pm}}:C\to\mb R$ by
$$F_{\infty_{\pm}}(Q)=
\begin{cases}
\dsp \frac{1}{\sqrt{1+x(Q)\bar{x(Q)}}} &\mbox{ if $x$ is a coordinate around $U_{Q}$ for $Q\in C_{0}$}\\
\dsp \frac{|z(Q)|}{\sqrt{1+z(Q)\bar{z(Q)}}} &\mbox{ if $z$ is a coordinate around $U_{Q}$ for $Q=\infty_{\pm}$}.
\end{cases}$$
\end{ex}
Let $\iota:X\to X$ be the hyperelliptic involution and extend $\iota$ to a group homomorphism on $\Div(X)$ by sending $D$ to $\iota(D)=\sum_{P\in X}n_{P}\iota(P).$ For any divisor $D=\sum_{P\in X}n_{P}P$ on $X,$ we set
$$F_{D}=\prod_{P\in X}F_{P}^{n_{P}}.$$
One sees that $F_{D}$ is smooth and positive on $X\setminus (\supp(D)\cup\supp(\iota(D))).$ When $D$ is effective, $F_{D}:X\to\mb R$ is a globally defined nonnegative continuous function on $X$ whose zero set coincides with $\supp(D)\cup\supp(\iota(D)).$ It follows from the definition that $F_{\iota(P)}=F_{P}$ for any $P\in X.$ By induction, $F_{\iota(D)}=F_{D}$ for any (effective) divisor $D$ of $X.$ 
Notice that the hyperelliptic involution has $2g+2$ fixed points which are precisely the Weierstrass points of $X;$ thus $\iota(W)=W$ where $W=\sum_{i=1}^{2g+2}P_{i}$ is the Weierstrass divisor on $X.$ 

Let $\mb Z_{2}$ be the subgroup of the automorphism group $\Aut(X)$\footnote{We consider the conformal automorphism group of $X.$} of $X$ generated by the hyperelliptic involution $\iota.$ Define a $\mb Z_{2}$ action on $C(X,\mb C)$ by $(\iota\cdot f)(P)=f(\iota(P))$ for $P\in C.$ A (continuous) function $f:X\to\mb C$ is said to be $\mb Z_{2}$-invariant if $\iota \cdot f=f.$ The set $C(X,\mb C)^{\mb Z_{2}}$ of all $\mb Z_{2}$-invariant complex valued continuous functions on $X$ forms a unital complex subalgebra of $C(X,\mb C).$ Since the quotient space $X/\mb Z_{2}$ is homeomorphic to $\mb P^{1},$ the algebra $C(X,\mb C)^{\mb Z_{2}}$ of $\mb Z_{2}$ invariant function is isomorphic to the algebra of complex valued continuous functions $C(\mb P^{1},\mb C)\cong C(X/\mb Z_{2},\mb C)$ in the category of $C^{*}$-algebra.

Let $\Div^{+}(X)$ be the set of all effective divisors on $X$ and $\Eff(X)$ be the vector subspace of $C(X,\mb C)$ spanned by $\{1,F_{D}:D\in\Div^{+}(X)\}.$ An element of $\Eff(X)$ can be represented as a sum $a_{0}+\sum_{D\in\Div^{+}(X)}a_{D}F_{D}$ where $a_{0}$ and $a_{D}$ are complex numbers such that $a_{D}=0$ for all but finitely many effective divisors $D.$ Since $F_{D+D'}=F_{D}F_{D'}$ for any $D,D'\in\Div^{+}(X),$ 
\begin{align*}
&\left(a_{0}+\sum_{D\in\Div^{+}(X)}a_{D}F_{D}\right)\left(b_{0}+\sum_{D'\in\Div^{+}(X)}b_{D'}F_{D'}\right)\\
&=a_{0}b_{0}+\sum_{D\in\Div^{+}(X)}b_{0}a_{D}F_{D}+\sum_{D'\in\Div^{+}(X)}a_{0}b_{D'}F_{D'}+\sum_{D,D'\in \Div^{+}(X)}a_{D}b_{D'}F_{D+D'}.
\end{align*}
This proves that $\Eff(X)$ forms a unital complex subalgebra of $C(X,\mb C).$ Since $F_{D}$ are all real valued, $F_{D}^{*}=F_{D}$ for any $D\in\Div^{+}(X).$ This implies that $\Eff(X)$ is a $*$-subalgebra of $C(X,\mb C).$ 

For each nonconstant element $F=a_{0}+\sum_{D\in\Div^{+}(X)}a_{D}F_{D}$ of $\Eff(X),$ we define the degree of $F$ to be $\deg F=\max\{\deg D:a_{D}\neq 0\}.$ It follows from the definition that $\deg F_{D}=\deg D$ for any $D\in\Div^{+}(X).$ When $F=a_{0}$ is a nonzero constant, we set $\deg F=0.$ For each nonnegative integer $i,$ we denote by $\Eff_{i}(X)=\{F\in\Eff(X):\deg F\leq i\}.$ Then $\{\Eff_{i}(X):i\geq 0\}$ forms a filtration for $\Eff(X)$ such that $\Eff(X)$ becomes a filtered algebra over $\mb C.$

\begin{thm}
The complex unital commutative filtered $*$-subalgebra $\Eff(X)$ of $C(X,\mb C)$ forms a dense subalgebra of $C(X,\mb C)^{\mb Z_{2}}$ with respect to the infinity-norm.\footnote{The infinity norm of a complex continuous valued function $f$ on a compact Hausdorff space $X$ is defined to be $\|f\|_{\infty}=\sup_{x\in X}|f(x)|.$} 
Hence the closure of $\Eff(X)$ in $C(X,\mb C)$ is isomorphic to $C(\mb P^{1},\mb C)$ in the category of $C^{*}$-algebra.
\end{thm}
\begin{proof}

By definition $\iota \cdot F_{P}=F_{P}$ for any $P\in X$ and hence $F_{P}$ is $\mb Z_{2}$-invariant. We can prove by induction that $F_{D}$ are all $\mb Z_{2}$-invariant functions. Since the generating set $\{1,F_{D}:D\in \Div^{+}(X)\}$ of $\Eff(X)$ consists of $\mb Z_{2}$-invariant functions, $\Eff(X)\subseteq C(X,\mb C)^{\mb Z_{2}}.$ To show that the closure of $\Eff(X)$ in $C(X,\mb R)$ coincides with $C(X,\mb C)^{\mb Z_{2}},$ we use the Stone-Weierstrass Theorem. At first, we identify $C(X,\mb C)^{\mb Z_{2}}$ with $C(\mb P^{1},\mb C)$ and identify $F_{P}$ with the function $f_{P}:\mb P^{1}\to\mb C$ by $f_{P}([Q])=F_{P}(Q)$ where $Q$ is any representative of $[Q]\in\mb P^{1}=X/\mb Z_{2}.$ For any $[P]\neq [Q]$ on $\mb P^{1},$ we choose a representative $P$ of $[P]$ and take the function $f_{P}:\mb P^{1}\to\mb C$ defined above. Since $[P]\neq [Q],$ $x(P)\neq x(Q)$ or $z(P)\neq z(Q).$ Then $f_{P}([P])=0\neq f_{P}([Q]).$ This shows that the algebra $\Eff(X)$ separates points of $\mb P^{1}.$ Since $\mb P^{1}$ is a compact Hausdorff space with respect to its complex analytic topology, by Stone-Weierestrass Theorem, $\Eff(X)$ is dense in $C(\mb P^{1},\mb C)$ with respect to the infinity norm.
\end{proof}

We remark that this proposition implies the Gelfand spectrum of the commutative $C^{*}$-subalgebra $\bar{\Eff(X)}$ of $C(X,\mb C)$ is $\mb P^{1}$ and that when the genus of the hyperelliptic curve is two, the Gaussian curvature function $K=-2F_{W}$ is an element of the algebra $\Eff(X).$ In fact, such a complex algebra $\Eff(X)$ can be defined using the notion of metrized effective divisors on any smooth projective variety $X.$ The general study of the algebra $\Eff(X)$ for a smooth complex projective variety $X$ will be given elsewhere.\\

\begin{thm}\label{p1}
Let $u_{P}:X\setminus \{P,\iota(P)\}\to\mb R$ be the function defined by $u_{P}=\log F_{P}$ for each $P\in X.$ Then $u_{P}:X\setminus \{P,\iota(P)\}\to\mb R$ is smooth and invariant under $\iota$ with $u_{\iota(P)}=u_{P}$ such that $\Delta u_{P}$ defines a continuous linear functional on the space $C^{\infty}(X,\mb C)$ of complex valued smooth functions on $X$ (with respect to the Frechet topology) by
$$\Delta u_{P}:C^{\infty}(X,\mb C)\to\mb C,\quad\varphi\mapsto\lim_{\epsilon\to 0+}\int_{X\setminus (D_{\epsilon}(P)\cup D_{\epsilon}(\iota(P)))}u_{P}(x)\Delta \varphi(x)d\nu(x)$$
such that
\begin{equation}\label{eqb}
\Delta u_{P}=\dsp\frac{1}{g+1}K-\frac{4}{g+1}F_{W}+2\pi(\delta_{P}+\delta_{\iota(P)})
\end{equation}
Here $D_{\epsilon}(Q)$ is a closed subset of $X$ containing $Q$ isomorphic to the $\epsilon$-closed disk $\{z\in\mb C:|z|\leq \epsilon\}$ in $\mb C.$
\end{thm}
\begin{proof}
Let $Q$ be any point on $X\setminus\{P,\iota(P)\}.$ Without loss of generality, we may assume that $Q\in C_{0}$ and $P\neq\infty_{\pm}.$ Choose an open neighborhood of $U_{Q}$ in $X\setminus\{P,\iota(P)\}$ such that $x$ defines a local coordinate on $U_{Q}.$ (For the cases when $Q\in C_{0}'$ and $P=\infty_{\pm},$ the computations are the same.) On $U_{Q},$ $u_{P}$ has the local expression
$$u_{P}(x)=\log F_{P}(x)=\log|x-x(P)|-\frac{1}{2g+2}\left(\log\sigma(x)+2\log(1+x\bar{x})\right).$$
Here $\sigma(x)=\sigma_{g}(x\bar{x})$ on $U_{Q}.$ On $U_{Q},$ $|x-x(P)|$ is positive and hence $\log|x-x(P)|$ is harmonic on $U_{Q}$. We obtain that
\begin{align*}
\Delta u_{P}(x) &= \frac{|f(x)|}{\sigma(x)}\frac{4\partial^{2}}{\partial x\partial\bar{x}}u_{P}(x)\\
                                &=\frac{4|f(x)|}{\sigma(x)}\cdot\frac{-1}{2g+2}\cdot\left(\frac{\sigma_{x\bar{x}}(x)\sigma(x)-\sigma_{x}(x)\sigma_{\bar{x}}(x)}{\sigma^{2}(x)}+\frac{2}{(1+x\bar{x})^{2}}\right)
\end{align*}
for any $x\in U_{Q}.$ On the other hand, the Gaussian curvature function has the local expression
$$K(x)=-2\frac{|f(x)|}{\sigma^{3}(x)}\cdot (\sigma_{x\bar{x}}(x)\sigma(x)-\sigma_{x}(x)\sigma_{\bar{x}}(x))$$
and $F_{W}(x)=|f(x)|/\sigma(x)(1+x\bar{x})^{2}$ on $U_{Q}.$ It follows from the definition that on $U_{Q},$
\begin{equation}\label{d1}
\Delta u_{P}(x)=\frac{K(x)}{g+1}-\frac{4}{g+1}F_{W}(x).
\end{equation}
Therefore (\ref{d1}) holds on $X\setminus \{P,\iota(P)\}.$ In a neighborhood $U_{P}$ of $P,$ the local function $|x-x(P)|$ has a zero at $P$ and hence $\log|x-x(P)|$ has a singularity at $P.$ To deal with the singularity of $\log|x-x(P)|,$ let us do the following analysis. We will do this analysis when $P$ is a Weierstrass point of $X.$ When $P$ is not a Weierstrass point of $X,$ the computation is similar and is left to the reader.

Let $P=P_{i}$ be a Weierstrass point of $X$ for $1\leq i\leq 2g+2.$ We choose $\epsilon>0$ small enough so that $\zeta=\sqrt{x-e_{i}}$ defines a local coordinate on $U_{\epsilon}(P)=\{Q\in C:|\zeta(Q)|<2\epsilon\}.$ Let $D_{\epsilon}(P)$ be the closed subset of $U_{\epsilon}(P)$ consisting of points $Q\in C$ so that $|\zeta(Q)|\leq \epsilon.$ By the Green's second identity,
\begin{equation}\label{green}
\int_{X\setminus D_{\epsilon}(P)}(u_{P}(x)\Delta \varphi(x)-\Delta u_{P}(x)\varphi(x))d\nu(x)=\oint_{\partial D_{\epsilon}(P)}\left(u_{P}(x)\frac{\partial\varphi}{\partial \mf n}(x)-\frac{\partial u_{P}}{\partial \mf n}(x)\varphi(x)\right)ds,
\end{equation}
where $ds$ is the line element of the real boundary curve $\partial D_{\epsilon}(P)$ and $\mf n$ is the outer unit normal vector field to $\partial D_{\epsilon}(P).$ On $X\setminus\{P\},$ (\ref{d1}) holds. We find that
$$\int_{X\setminus D_{\epsilon}(P)}\Delta u_{P}(x)\varphi(x)d\nu(x)=\int_{X\setminus D_{\epsilon}(P)}\left(\frac{K(x)}{g+1}-\frac{4}{g+1}F_{W}(x)\right)\varphi(x)d\nu(x).$$
By the continuities of $K$ and $F_{W}$ and $\varphi$ and the compactness of $X$ together with the Lebesgue dominated convergence theorem, we have
\begin{align*}
\lim_{\epsilon\to 0+}\int_{X\setminus D_{\epsilon}(P)}\Delta u_{P}(x)\varphi(x)d\nu
&=\lim_{\epsilon\to 0+}\int_{X}\left(\frac{K(x)}{g+1}-\frac{4}{g+1}F_{W}(x)\right)\varphi(x)\chi_{X\setminus D_{\epsilon}}(x)d\nu(x)\\
&=\int_{X}\left(\frac{K(x)}{g+1}-\frac{4}{g+1}F_{W}(x)\right)\varphi(x)d\nu(x).
\end{align*}
Here $\chi_{A}:X\to\mb C$ denotes the characteristic function of a (Borel measurable) subset $A$ of $X.$ To compute the limit of the right hand side of (\ref{green}) as $\epsilon$ tends to $0+,$ we use the local properties of $u_{P}.$ On $U_{\epsilon}(P),$ the function $u_{P}$ has the local expression:
$$u_{P}(\zeta)=2\log|\zeta|+\alpha(\zeta)$$
for some smooth function $\alpha:U_{\epsilon}(P)\to\mb R.$ By compactness of the real curve $\partial D_{\epsilon}(P)$ and the continuities of $\alpha$ and $\nabla\varphi,$ there exists $M>0$ such that
$$\left|\oint_{\partial D_{\epsilon}(P)}\alpha (\zeta)\frac{\partial\varphi}{\partial \mf n}ds\right|\leq M\epsilon.$$
Similarly, we can find $M'>0$ so that
$$\left|\oint_{\partial D_{\epsilon}(P)}2\log|\zeta|\frac{\partial\varphi}{\partial \mf n}ds\right|\leq M'|\epsilon\log\epsilon|.$$
By $\lim_{\epsilon\to 0+}\epsilon\log\epsilon=0$ and $\lim_{\epsilon\to 0+}\epsilon=0,$ we see that
$$\lim_{\epsilon\to 0+}\oint_{\partial D_{\epsilon}(P)}u_{P}\frac{\partial\varphi}{\partial \mf n}ds=0.$$
Let us compute
$$\oint_{\partial D_{\epsilon}}\frac{\partial u_{P}}{\partial \mf n}\varphi ds=\oint_{\partial D_{\epsilon}}\varphi\frac{\partial}{\partial \mf n}\left(2\log|\zeta|+\alpha(\zeta)\right)ds.$$
Again, by compactness of $\partial D_{\epsilon}(P)$ and the continuities of $\varphi$ and $\nabla \alpha,$ we can find $M''$ such that
$$\left|\oint_{\partial D_{\epsilon}(P)}\varphi\frac{\partial \alpha}{\partial \mf n}ds\right|\leq M''\epsilon.$$
By taking $\epsilon \to 0+,$ the above integral converges to $0.$ 

Let $\gamma:[0,2\pi]\to \partial D_{\epsilon}(P)$ be the parametrization of $\partial D_{\epsilon}(P)$ defined by $\gamma(t)=\zeta^{-1}(\epsilon e^{it})$ for $t\in [0,2\pi].$ The arclength function of $\partial D_{\epsilon}(P)$ is given by
$$s(t)=\int_{0}^{t}\epsilon\sqrt{\rho_{i}(\gamma(u))}du,\quad t\in [0,2\pi]$$
and the outernomal derivative of $2\log|\zeta|$ along $\partial D_{\epsilon}(P)$ at $\gamma(t)$ is $-2/\epsilon\sqrt{\rho_{i}(\gamma(t))}.$ Here $\rho_{i}$ is the conformal factor of the metric $ds^{2}$ around $P=P_{i}$ defined in Section \ref{mf}. 
By simple calculation,
\begin{align*}
\oint_{\partial D_{\epsilon}(P)}\varphi(\zeta)\left(2\frac{\partial}{\partial\mf n}\log|\zeta|\right)ds&=-2\int_{0}^{2\pi}\varphi\circ\zeta^{-1}(\epsilon e^{it})\frac{1}{\epsilon\sqrt{\rho_{i}(\gamma(t))}}\cdot \epsilon\sqrt{\rho_{i}(\gamma(t))} dt\\
&=-2\int_{0}^{2\pi}\varphi\circ\zeta^{-1}(\epsilon e^{it})dt.
\end{align*}
By continuities of $\varphi$ and $\zeta,$ we find
$$\lim_{\epsilon\to 0+}\int_{0}^{2\pi}\varphi\circ\zeta^{-1}(\epsilon e^{it})dt=\varphi\circ\zeta^{-1}(0)\cdot 2\pi.$$
Since $\zeta^{-1}(0)=P,$ we obtain that
$$\lim_{\epsilon\to 0+}\oint_{\partial D_{\epsilon}(P)}\varphi(\zeta)\left(2\frac{\partial}{\partial\mf n}\log|\zeta|\right)ds=-4\pi\varphi(P).$$
We conclude that
$$\lim_{\epsilon\to 0+}\oint_{\partial D_{\epsilon}(P)}\left(u_{P}(x)\frac{\partial\varphi}{\partial \mf n}(x)-\frac{\partial u_{P}}{\partial \mf n}(x)\varphi(x)\right)ds=4\pi \varphi(P).$$
As a consequence,
\begin{align*}
\lim_{\epsilon\to 0+}\int_{X\setminus D_{\epsilon}(P)}u_{P}(x)\Delta\varphi (x) d\nu(x)&=\int_{X}\left(\frac{K(x)}{g+1}-\frac{4}{g+1}F_{W}(x)\right)\varphi(x)d\nu(x)+4\pi\varphi(P)
\end{align*}
when $P=P_{i}$ is a Weierstrass point of $X.$ This shows that
$$\Delta u_{P}(x)=\frac{K(x)}{g+1}-\frac{4}{g+1}F_{W}(x)+4\pi\delta_{P}$$
as a continuous linear functional on $C^{\infty}(X,\mb C)$ when $P$ is a Weierstrass point. 
\end{proof}

For any divisor $D=\sum_{P\in X}n_{P}P$ on $X,$ we define a distribution $$\delta_{D}:C^{\infty}(X,\mb C)\to\mb C\mbox{ by }\delta_{D}(\varphi)=\sum_{P\in X}n_{P}\varphi(P).$$
One sees that $\delta_{D}=\sum_{P\in X}n_{P}\delta_{P}$ as a distribution on $X$ and the map $$\delta:\Div(X)\to C^{\infty}(X,\mb C)'$$ sending $D$ to $\delta_{D}$ is a group homomorphism. Define a group endomorphism on $\Div(X)$$$\alpha:\Div(X)\to\Div(X)\mbox{ by }\sum_{P\in X}n_{P}P \mapsto \sum_{P\in X}n_{P}(P+\iota(P)).$$ 
Equation (\ref{eqb}) can be rewritten as
\begin{equation}\label{simb1}
\Delta u_{P}=\dsp\frac{1}{g+1}K-\frac{4}{g+1}F_{W}+2\pi\delta_{\alpha(P)}\mbox{ for any $P\in X$}.
\end{equation}
When the genus of $X$ is two, $K=-2F_{W}$ and hence Equation (\ref{simb1}) can be simplified into the form
\begin{equation}\label{simb2}
\Delta u_{P}=-2F_{W}+2\pi\delta_{\alpha(P)}
\end{equation}
for each $P\in X.$ 


For any divisor $D\in \Div(X),$ we define a smooth function
$$u_{D}:X\setminus(\supp(\alpha(D)))\to\mb R,\mbox{ by $u_{D}=\log F_{D}.$}$$
It follows from the definition that $u_{\iota(D)}=u_{D}$ for any $D\in\Div(C)$ and $$u_{D+D'}=u_{D}+u_{D'}\mbox{ on $X\setminus \supp(\alpha(D+D')).$}$$ 
If we let $\Xi(X)$ to be the set of all $\mb Z$-linear combinations of $\{u_{P}:P\in X\},$ then $u_{D}$ belongs to $\Xi(X)$ for any divisor $D$ of $X$ and the map $$u:\Div(X)\to \Xi(X),\quad D\mapsto u_{D}$$ 
is an abelian group homomorphism.

\begin{cor}\label{p2}
For any divisor $D$ of $X,$ $u_{D}$ is smooth on $X\setminus\supp(\alpha(D))$ so that $\Delta u_{D}$ defines a continuous linear functional on $C^{\infty}(X,\mb C)$ with
$$\Delta u_{D}=\dsp\frac{\deg D}{g+1}K-\frac{4\deg D}{g+1}F_{W}+2\pi\delta_{\alpha(D)}$$
When the genus of $X$ is two, the above equation can be simplified to
$$\Delta u_{D}=(\deg D)K+2\pi\delta_{\alpha(D)}.$$
\end{cor}

The Corollary \ref{p2} implies that
\begin{cor}
Let $D$ be a divisor of degree $0$ on $X.$ Then $\Delta u_{D}=2\pi \delta_{\alpha(D)}.$
\end{cor}

The first equation in Corollary \ref{p2} follows directly from Equation (\ref{simb1}) and the additive properties of group homomorphism $u:\Div(X)\to\Xi(X)$ and the second equation in Corollary \ref{p2} follows from (\ref{simb2}).  Furthermore, when the genus of $X$ is two, $e^{u_{W}}=F_{W}$ and $\alpha(W)=2W.$ In this case, we obtain the following equation
$$\Delta u_{W}+2(\deg W)e^{u_{W}}=4\pi\delta_{W}$$
which coincides with the result in Theorem \ref{thm2}. 

\begin{thm}
Let $\{Q_{1},\cdots,Q_{s}\}$ be a set of $s$ distinct points of $X$ and $\{m_{1},\cdots,m_{s}\}$ be a set of $s$ positive integers. Denote the divisor $\sum_{i=1}^{s}m_{i}Q_{i}$ by $D.$ Assume that $\{Q_{1},\cdots,Q_{s}\}$ contains no Weierstrass points of $X.$ Then the following equation 
\begin{equation}\label{MFE3}
\Delta v+4\rho(D)F_{D}e^{v}=\rho(D)K+2\pi\delta _{\alpha(W-D)}.
\end{equation}
has a solution defined and smooth on $X\setminus \{Q_{1},\cdots,Q_{s},P_{1},\cdots,P_{2g+2}\},$ where $\rho(D)=\deg (W-D)/(g+1).$ 
\end{thm}
\begin{proof}
Then $D$ is an effective divisor on $X.$ It is not difficult to see that the function $v=u_{W-D}$ satisfies the following differential equation (\ref{MFE3}).

\end{proof}

Now we are ready to prove Theorem \ref{thmeff}. When the genus of $X$ is two, Equation (\ref{MFE3}) turns into the following differential equation 
\begin{equation}\label{MFE4}
\Delta v+2\deg (W-D)F_{D}e^{v}=2\pi\delta_{\alpha(W-D)}.
\end{equation}
Notice that Equation (\ref{MFE4}) is equivalent to Equation (\ref{MFELW2}). We complete the proof of Theorem \ref{thmeff}.

It is not clear to us, at this moment, how to construct all solutions to (\ref{MFEF}) or to (\ref{MFEF2}). It will be our future study to find the solution space to those mean field equations on complex hyperelliptic curves.\\

\bibliography{sampartb}

\end{document}